\documentclass[letterpaper, 11pt]{article}
\usepackage{amsmath, amsthm, amssymb}
\usepackage[numbers,sort]{natbib}
\usepackage{bm}
\usepackage{graphicx}
\usepackage{fullpage}
\usepackage[colorlinks,linkcolor=blue,citecolor=blue,urlcolor=magenta,linktocpage,plainpages=false]{hyperref}\usepackage{xcolor}
\usepackage{xspace}
\usepackage{url}
\usepackage[noend]{algpseudocode}
\usepackage{libertine}
\usepackage{libertinust1math}
\usepackage[ruled,vlined]{algorithm2e}
\usepackage{caption}
\usepackage{setspace}
\usepackage[left=1in, right=1in, top=1in, bottom=1in]{geometry}
\usepackage{authblk}

\newtheorem{lemma}{Lemma}
\newtheorem{thm}{Theorem}
\newtheorem{cor}{Corollary}

\newcommand{\E}[0]{\ensuremath \mathbb{E}}
\newcommand{\R}[0]{\ensuremath \mathbb{R}}
\newcommand{\e}{\varepsilon}
\newcommand{\ip}[2]{\langle #1, #2 \rangle}

\newcommand{\Var}{\mathrm{Var}}
\newcommand{\ones}{\bm{1}}

\newcommand{\OPT}{\textrm{OPT}\xspace}
\newcommand{\vol}{\textrm{vol}}
\renewcommand{\c}{c}
\newcommand{\A}{A}
\newcommand{\T}{\mathcal{T}}
\newcommand{\I}{\mathcal{I}}
\newcommand{\LP}{\textrm{LP}}
\newcommand{\IP}{\textrm{IP}}

\newcommand{\poly}{\text{poly}}

\newcommand{\mom}[1]{{\left\vert\kern-0.25ex\left\vert\kern-0.25ex\left\vert #1 \right\vert\kern-0.25ex\right\vert\kern-0.25ex\right\vert}}

\makeatletter
\newcommand{\algmargin}{\the\ALG@thistlm}
\makeatother
\newlength{\whilewidth}
\algnewcommand{\parState}[1]{\State%
  \parbox[t]{\dimexpr\linewidth-\algmargin}{\strut #1\strut}}
\algtext*{Until}%

\newcounter{mynotes}
\newcommand{\bS}{\mathbb{S}}

\newcommand{\cst}{\text{cst}}

\newcommand{\pareto}{\Delta_{\text{rc}}}
\newcommand{\gap}{\textsc{IPGap}\xspace}

\DeclareMathOperator{\argmax}{argmax}
\DeclareMathOperator{\argmin}{argmin}

\title{Branch-and-Bound Solves Random Binary IPs in Polytime \footnote{The contributions of this manuscript overlap with those of \cite{dey2021branch}. The differences are noted in the final paragraph of Section \ref{sec:intro}.}}
\author[1]{Santanu S. Dey\thanks{santanu.dey@isye.gatech.edu}}
\author[1]{Yatharth Dubey\thanks{yatharthdubey7@gatech.edu}}
\author[2]{Marco Molinaro\thanks{molinaro@inf.puc-rio.br}}
\affil[1]{School of Industrial and Systems Engineering, Georgia Institute of Technology}
\affil[2]{Computer Science Department, Pontifical Catholic University of Rio de Janeiro}
\date{}

\begin{document}
\maketitle

\begin{abstract}
Branch-and-bound is the workhorse of all state-of-the-art mixed integer linear programming (MILP) solvers. These implementations of branch-and-bound typically use variable branching, that is, the child nodes are obtained by fixing some variable to $0$ in one node and to $1$ in the other node. Even though modern MILP solvers are able to solve very large-scale instances efficiently, relatively little attention has been given to understanding why the underlying branch-and-bound algorithm performs so well. In this paper, our goal is to theoretically analyze the performance of the standard variable branching based branch-and-bound algorithm. In order to avoid the exponential worst-case lower bounds, we follow the common idea of considering random instances. More precisely, we consider random integer programs where the entries of the coefficient matrix and the objective function are randomly sampled. 

Our main result is that with good probability branch-and-bound with variable branching explores only a polynomial number of nodes to solve these instances, for a fixed number of constraints. To the best of our knowledge this is the first known such result for a standard version of branch-and-bound. We believe that this result provides an 
indication as to why branch-and-bound with variable branching works so well in practice. 
\end{abstract}


\section{Introduction}\label{sec:intro}
The branch-and-bound algorithm, first  proposed by Land and Doig in~\cite{land1960automatic}, is the workhorse of all modern  state-of-the-art mixed integer linear programming (MILP) solvers. As is well known, the branch-and-bound algorithm searches the solution space by recursively partitioning it. The progress of the algorithm is monitored by maintaining a ``tree''. Each node of the tree corresponds to a linear program (LP) solved, and in particular, the root node corresponds to the LP relaxation of the integer program. After solving the LP corresponding to a node, the feasible region of the LP  is partitioned into two subproblems (which correspond to the child nodes of the given node), so that the fractional optimal solution of the LP is not included in either subproblem, but any integer feasible solution contained in the feasible region of the LP is included in one of the two subproblems. This is accomplished by adding an inequality of the form $\pi x \leq \pi_0$ to first subproblem and the inequality $\pi^{\top}x \geq \pi_0 + 1$ to the second subproblem, where $\pi$ is an integer vector and $\pi_0$ is an integer scalar. The process of partitioning at a node stops if (i) the LP at the node is infeasible, (ii) the LP's optimal solution is integer feasible, or (iii) the LP's optimal objective function value is worse than an already known integer feasible solution.  These three conditions are sometimes referred to as the rules for pruning a node. The algorithm terminates when there are no more ``open nodes" to process, i.e. all nodes have been pruned. A branch-and-bound algorithm is completely described by fixing a rule for partitioning the feasible region at each node and a rule for selecting which open node should be solved and branched on next.  See ~\cite{wolsey1999integer,conforti2014integer} for more discussion on the branch-and-bound algorithm. 

In 1983, Lenstra~\cite{lenstra1983integer} showed that general integer programs can be solved in polynomial time in fixed dimension. This algorithm, which is essentially a branch-and-bound algorithm, uses tools from geometry of numbers, in particular the lattice basis reduction algorithm~\cite{lenstra1982factoring} to decide on $\pi$ for partitioning the feasible region. Pataki et al.~\cite{pataki2010basis} proved that most random packing integer programs can be solved at the root-node using a partitioning scheme similar to the one proposed by Lenstra~\cite{lenstra1983integer}. While there are some implementations of such general partitioning rules~\cite{aardal2000market}, all state-of-the-art solvers use a much simpler (and potentially significantly weaker and restrictive) partitioning rule for solving binary IPs, namely:  If $x_j$ is fractional in the optimal solution of the LP of a given node, then one child node is obtained by the addition of the constraint $x_j \leq 0$ and the other with the constraint $x_j \geq 1$ (i.e., $\pi = e_j$, the unit vector in direction $j$ for some $j$ in $\{1, \dots, n\}$). The rule for deciding how to partition the feasible region at a node then reduces to choosing which fractional variable should be branched on. This kind of partitioning rule is henceforth referred to as \emph{variable branching}.

As mentioned above, all state-of-the-art MILP solvers use variable branching, which has proven itself to be very successful in practice~\cite{bixby2007progress}. Part of this success could be attributed to the fact that variable branching helps maintain the sparsity structure of the original LP relaxation, which can help in solving LPs in the branch-bound-tree faster (see~\cite{coleman1,walter2014sparsity,dey2015approximating,dey2018analysis}). Additionally, while in the worst-case there can be exponentially many nodes in the branch-and-bound tree  (see ~\cite{chvatal1980hard,dash2002exponential} for explicit examples for variable branching based branch-and-bound), a major reason for its success is that in practice the size of the tree can be quite small~\cite{steffyS17}. To the best of our knowledge there is no theoretical study of branch-and-bound algorithm using variable branching that attempts to explain its incredible success in practice. 

In order to avoid worst-case lower bounds, a standard idea is to consider random instances. A famous example is the study of smoothed analysis for the simplex method~\cite{spielman2004smoothed}. In this paper, we provide what seems to be the first analysis of the branch-and-bound algorithm with variable branching for a set of random instances. More precisely, we consider problems of the form 
  \begin{align}
        \max &\quad \ip{\c}{x} \notag\\
        \text{s.t.} &\quad \A x  \leq b   \tag{$\IP$}\\
        &\quad x \in \{0,1\}^n. \notag
    \end{align}
  By a \textbf{random instance} of $\IP$ we mean one where we draw the entries of the constraint matrix $A \in \mathbb{R}^{m \times n}$ ($m \ll n$) and the objective vector $c$ uniformly from $[0,1]$ independently.  For the right-hand-side, we will use $b_j = \beta_j \cdot n$, where $\beta_j \in (0, \frac{1}{2})$ for $j \in \{1, \dots, m\}$. (The case when $\beta_j$ is high is less interesting since all items fit with probability 1 as $n \rightarrow \infty$ and $m$ is fixed.) We show that if the number of constraints $m$ is fixed, then the branch-and-bound tree with variable branching has at most polynomial number of nodes with good probability. More precisely, we show the following result.

	\begin{thm} \label{thm:main}
			Consider a branch-and-bound algorithm using the following rules:
\begin{itemize}
\item Partitioning rule: Variable branching, where any fractional variable can be used to branch on.
\item Node selection rule:  Select a node with largest LP value as the next node to branch on. 
\end{itemize}
	Consider $n \ge m+1$ and a random instance of the problem $\IP(b)$ where $b_j = \beta_j \cdot n$ and $\beta_j \in (0,1/2)$ for $j \in \{1,\ldots,n\}$. Then with probability at least $1 - \frac{1}{n} - 2^{-\alpha \bar{a}_2}$, the branch-and-bound algorithm applied to this random instance produces a tree with at most $$n^{\bar{a}_1 \cdot (m + \alpha \log m)}$$ nodes for all $\alpha \le \min\{30 m\,,\, \frac{\log n}{\bar{a}_2}\}$, where $\bar{a}_1$ and $\bar{a}_2$ are constant depending only on $m$ and $\beta$.
		\end{thm}	
		
We note that the node selection rule used here is called the \emph{best-bound method} in the literature and often used in practice with minor modifications~\cite{linderoth1999computational}. This node selection rule is known to guarantee the smallest tree for any fixed partitioning rule~\cite{wolsey1999integer}. Also notice that Theorem~\ref{thm:main} does not specify a rule for selecting a fractional variable to branch on and therefore even ``adversarial'' choices lead to a polynomial sized tree with good probability. This indicates that the tree is likely to be even smaller when one uses a ``good'' variable selection rule, such as \emph{strong branching}~\cite{achterberg2005branching}. Another reason for the size of the tree to be even smaller in practice is
that Theorem \ref{thm:main} relies only on rule (iii) of pruning, i.e., pruning by bounds, to bound the size of the tree. However, pruning may also occur due to rules (i) (LP infeasibility) or (ii) (integer optimality), thus leading to a
smaller tree size than predicted by Theorem~\ref{thm:main}. 

Also notice that while $\IP$ is written in packing form (e.g., $A$ is non-negative), our bounds work for every deterministic binary IP that is ``well-behaved'', as discussed in Section \ref{sec:remark}. Together with the above observations, we believe Theorem~\ref{thm:main} provides compelling indication of why branch-and-bound with variable branching works so well in practice.   

Finally, we note that random (packing) instances have been considered in several previous studies, and it has been shown that they can be solved in polynomial time with high probability. As mentioned earlier, Pataki et al.~\cite{pataki2010basis} proves that random packing integer programs can be solved in polynomial-time; however it uses the very heavy machinery of lattice basis reduction, which is not often used in practice. 
Other papers considering random packing problems present algorithms that are custom-made enumeration-based schemes that are \textbf{not} equivalent to the general purpose branch-and-bound algorithm. In particular, Lueker~\cite{lueker} showed that the additive integrality gap for random one-row (i.e. $m = 1$) knapsack instances is $O\left(\frac{\log^2 n}{n}\right)$, and using this property Goldberg and Marchetti-Spaccamela~\cite{goldbergSpaccamela} presented a polynomial-time enumeration algorithm for these instances. Beier and Vocking~\cite{beierVockingKnapSTOC,beierVockingKnapSODA} showed that the so-called knapsack core algorithm with suitable improvements using enumeration runs in polynomial-time with high probability. Finally, Dyer and Frieze~\cite{dyerFriezeKnap} generalize the previous results on integrality gap and enumeration techniques to the random instances we consider here.  

There has also been follow-up work to the results presented here. Borst et al. \cite{borst2021integrality} show that, when the data is drawn from a standard Gaussian (as opposed to the uniform distribution studied here), the integrality gap is sufficiently small that branch-and-bound solves problems of this type in $n^{\poly(m)}$ nodes. Frieze \cite{frieze2020expected} shows that any branch-and-bound tree solving the asymmetric travelling salesperson problem via the assignment problem relaxation has an exponential number of nodes.





The rest of the paper is organized as follows. Section~\ref{sec:pre} presents notation, formalizes the set-up and presents some preliminary results needed. Section~\ref{sec:goodsol} establishes a key result that the size of branch-and-bound tree  can be bounded if one can bound all possible ``near optimal" solutions for $\IP(b)$ for varying values of $b$. Section~\ref{sec:geometry} and Section~\ref{sec:distance} build up machinery to prove that the number of ``near optimal" solutions is bounded by a polynomial. Section~\ref{sec:prthm} completes the proof of Theorem~\ref{thm:main}. {\color{black} In Section~\ref{sec:remark}, we present an upper bound on the size of branch-and-bound tree for general IPs.}

{\color{black} An earlier version of this paper is published in the conference SODA~\cite{dey2021branch}. The current paper is a significant improvement over~\cite{dey2021branch}.  Although the main result (Theorem~\ref{thm:main}) is the same, we have significantly simplified the proofs and sharpened the key take-away result in Section~\ref{sec:remark} that can be applied to obtain upper bounds on the size of branch-and-bound tree for general IP models: see Corollary~\ref{cor:final_rem}. Comparing with~\cite{dey2021branch}, we improve this result by reducing the upper bound on size of branch-and-bound trees by a factor of  $O(n^{m})$. Also the analysis in Section~\ref{sec:goodsol} and Section~\ref{sec:geometry} are revamped and simplified as compared the original paper~\cite{dey2021branch}.}

	\section{Preliminaries}\label{sec:pre}
	    
  \subsection{Branch-and-bound}

	Even though the general branch-and-bound algorithm was already described in the introduction, we describe it again here for maximization-type 0/1 IPs and using variable branching as partitioning rule and best-bound as node selection rule, for a clearer mental image. This is what we will henceforth refer to as \emph{the branch-and-bound} (BB) algorithm. 
	
	The algorithm constructs a tree $\T$ where each node has an associated LP; the LP relaxation of original integer program is the LP of the root node. In each iteration the algorithm:
	
	\begin{enumerate}
		\item (Node selection) Selects an unpruned leaf $N$ with highest optimal LP value in the current tree $\T$, and obtains an optimal solution $\tilde{x}$ to this LP.
		
		\item (Pruning by integrality) If $\tilde{x}$ is integral, and hence feasible to the original IP, and has higher value than the current best such feasible solution, it sets $\tilde{x}$ as the current best feasible solution. The node $N$ is marked as \emph{pruned by integrality}.
		
		\item (Pruning by infeasibility/bound) Else, if the LP is infeasible or its value is worse than the value of the current best feasible solution, the node $N$ is marked as \emph{pruned by infeasibility/bound}.
		
		\item (Branching) Otherwise it selects a coordinate $j$ where $\tilde{x}$ is fractional and adds two children to $N$ in the tree: on one of them it adds the constraint $x_j = 0$ to the LP of $N$, and on the other it adds the constraint $x_j = 1$ instead.   
	\end{enumerate}

 One simple but important property of this best-bound node selection rule is the following. (Note that we do not assume that an optimal IP solution is given in the beginning of the procedure: the algorithm starts with no current best feasible solution, which are only found in step \emph{Pruning by integrality}.)

    \begin{lemma}[best-bound node selection] \label{lemma:pruning}
    	The execution of branch-and-bound with best-bound node selection rule never branches on a node whose LP value is worse than the optimal value of the original IP.
    \end{lemma}
    
    \begin{proof}
    	Let $\IP^*$ denote the optimal value of the original IP. Notice that throughout the execution, either:
    	
    	\begin{enumerate}
    		\item  The current best feasible solution has value $\IP^*$ (i.e. an optimal integer solution has been found)
    		
    		\item The LP of an unpruned leaf contains an optimal integer solution, and so this LP value is at least as good as $\IP^*$. 
    	\end{enumerate}
    	This means that in every iteration, the algorithm cannot \emph{select} and \emph{branch} on a leaf with LP value strictly less than $\IP^*$: in Case 2 such leaf is not selected (due to the best-bound rule), and in Case 1 such leaf is pruned by infeasibility/bound (and hence not branched on) if selected. 
    \end{proof}
    

	\subsection{(Random) Packing problems}

 We will use the following standard observation on the number of fractional coordinates in an optimal solution of the LP relaxation of every instance of $\IP$ (see for example Section 17.2 of \cite{vazirani}).
 	
	\begin{lemma}\label{lemma:basicSol}
    Consider an instance of $\IP$. Then every LP in the BB tree for this instance has an optimal solution with at most $m$ fractional coordinates.
  \end{lemma}
  
  \begin{proof}
		Notice that the LP's in the BB tree for this instance have the form 
	\begin{align}
    \max &\quad \ip{\c}{x} \notag\\
    \text{s.t.} &\quad \A x  \leq b \notag\\
    &\quad x_j = 0, ~~~~~~~~\forall j \in J_0 \label{eq:basic}\\
    &\quad x_j = 1, ~~~~~~~~\forall j \in J_1 \notag\\
    &\quad x_j \in [0,1], ~~\forall j \notin J_0 \cup J_1, \notag
	\end{align}
	for disjoint subsets $J_0,J_1 \subseteq [n]$ (i.e. the fixings of variables due to the branchings up to this node in the tree).
	
  	The feasible region $P$ in \eqref{eq:basic} is bounded, so there is an optimal solution $x^*$ of the LP that is a vertex of $P$. This implies that at least $n$ of the constraints of the LP are satisfied by $x^*$ at equality. Since there are $m$ constraints in $Ax \le b$, at least $n-m$ of these equalities are of the form $x^*_j = v_j$ (for some $v_j \in \{0,1\}$) and so at most $m$ $x^*_j$'s can be fractional. 
  \end{proof}
	
	Recall that the \emph{integrality gap} of $\IP$ is $\gap := \OPT(\LP) - \OPT(\IP)$, namely the  optimal value its LP relaxation $\LP := \max\{\ip{c}{x} : Ax \le b, ~x \in [0,1]^n\}$ minus the optimal IP value. Dyer and Frieze proved the following property that will be crucial for our results: for a random instance of $\IP$ (defined right after the definition of $\IP$), the integrality gap is $O(\frac{\log^2 n}{n}$). 	
	
	\begin{thm}[Theorem 1 of~\cite{dyerFriezeKnap}] \label{thm:gap}
		Consider a vector $\beta \in (0,\frac{1}{2})^n$ and let $b \in \R^n$ be given by $b_j = \beta_j \cdot n$ for $j \in \{1,\ldots,n\}$. Then there are scalars $a_1,a_2 \ge 1$ depending only on $m$ and $\min_j \beta_j$ such that the following holds: for a random instance $\mathcal{I}$ of $\IP$ 
		\begin{align*}
			\Pr\bigg(\gap(\mathcal{I}) \ge  \alpha a_1\,\frac{\log^2 n}{n}\bigg) ~\le~ 2^{-\alpha a_2}
		\end{align*}
		for all $\alpha \le \frac{3 \log n}{a_2}$.
	\end{thm}


	\subsection{Notation}
	
	We use the shorthands $ \binom{n}{\leq k} := \sum_{i = 0}^k \binom{n}{i}$ and $[n] := \{1,2,\ldots,n\}$. We also use $\binom{[n]}{\leq k}$ to denote the family of all subsets of $[n]$ of size at most $k$. We use $A^j$ to denote the $j$th column of the matrix $A$. For any vector $x \in \mathbb{R}^n$, we use $x^+$ to denote the vector that satisfies $x^+_i = \max(0, x_i)$ for all $i \in [n]$.


	\section{Branch-and-bound and good integer solutions}\label{sec:goodsol}
	
		In this section we connect the size of the BB tree for any instance of $\IP$ and the number of its near-optimal solutions. 
	
     %
     %

To make this precise, first let $\LP$ be the LP relaxation of $\IP$ 
    \begin{align}
       \max &\quad \ip{\c}{x} \notag\\
        \text{s.t.} &\quad \A x  \leq b   \tag{$\LP$}\\
        &\quad x \in [0,1]^n, \notag
    \end{align}
    and its Lagrangian relaxation
    \begin{align*}
       \min_{\lambda \ge 0} \max_{x \in [0,1]^n} \ip{\c}{x} - \lambda^\top (A x - b) \quad\equiv\quad \min_{\lambda \ge 0} \max_{x \in [0,1]^n} \sum_{j \in [n]} (c_j - \ip{\lambda}{A^j}) x_j + \ip{b}{\lambda}.
    \end{align*}    
  Let $(x^*,\lambda^*)$ be a saddle point of this Lagrangian, namely  
  %
  \begin{align*}
  	x^* &= \argmax\bigg\{\sum_{j \in [n]} (c_j - \ip{\lambda^*}{A^j})x_j + \ip{b}{\lambda^*} ~:~ x \in [0,1]^n\bigg\}\\
  	\lambda^* &= \argmin\bigg\{\sum_{j \in [n]} (c_j - \ip{\lambda}{A^j})x^*_j + \ip{b}{\lambda} ~:~ \lambda \ge 0\bigg\}.
  \end{align*}
  Recall that this saddle point solution has the same value as $\OPT(\LP)$~\cite{ben2001lectures}. 
Also notice that $c_j - \ip{\lambda^*}{A^j}$ can be thought as the reduced cost of variable $x_j$ and by optimality of $x^*$
\begin{align}
    x^*_j = \begin{cases}
    1 & \text{ if  } c_j - \ip{\lambda^*}{A^j} > 0 \\
    0 & \text{ if  } c_j - \ip{\lambda^*}{A^j} < 0.
    \end{cases}   \label{eq:lin1}
\end{align}  
  
 	For any point $x \in [0,1]^n$, we define the quantity 
$$\pareto(x) = \sum_{j \in [n]} (c_j - \ip{\lambda^*}{A^j}) \cdot (x^*_j - x_j),$$
		which is then the difference in value between $x$ and $x^*$ given by the reduced costs. Given this, we say that a 0/1 point $x$ \emph{good} if its $\pareto$ is at most $\gap$, and we use $G$ to denote the set of all good points, namely 
	\begin{align}
	G := \{x \in \{0,1\}^n : \pareto(x) \le \gap\}. \label{eq:good}
	\end{align}

	The following is the main result of this section, which states that we can bound the size of the branch-and-bound tree based on the number of these good points. 
	
	\begin{thm}\label{thm:good}
	Consider the branch-and-bound algorithm with best-bound node selection rule for solving $\IP$. Then its final tree has at most $2 |G| n + 1$ nodes.
	\end{thm}	
	
		The remainder of this section is dedicated to proving this result. So let $\mathcal{T}$ denote the final BB tree constructed by the algorithm, and let $\mathcal{N}$ denote the set of its internal nodes. The first observation is that all LP solutions seen throughout the BB tree have small $\pareto$. 
		
	\begin{lemma}\label{lem:internal_pareto}
	Let $N$ be a node that is branched on in the BB tree $\mathcal{T}$, i.e. $N \in \mathcal{N}$, and let $x^N$ be an optimal solution for the LP of this node with at most $m$ fractional coordinates (via Lemma \ref{lemma:basicSol}). Then,
	$$\pareto(x^N) \leq \gap.$$
	\end{lemma}
	
	\begin{proof}
		First notice that for every feasible solution $x$ to $\LP$ we have 
		\begin{align*}
			\OPT(\LP) \ge \ip{c}{x} + \pareto(x),
		\end{align*} 
 since the equality of optimal value of $\LP$ and its Lagrangian and then feasibility of $x$ give
	\begin{align*}
		\OPT(\LP) - \ip{c}{x} &= \sum_{j \in [n]} (c_j - \ip{\lambda^*}{A^j})x^*_j + \ip{b}{\lambda^*} - \ip{c}{x} \ge \sum_{j \in [n]} (c_j - \ip{\lambda^*}{A^j})x^*_j + \ip{\A x}{\lambda^*} - \ip{c}{x},
	\end{align*}
	and the claim follows by noticing that the right-hand side is precisely $\pareto(x)$. Moreover, by the best-bound node selection rule (Lemma \ref{lemma:pruning}) we have that $\ip{c}{x^N} \ge \OPT(\LP) - \gap$. Putting these observations together proves the result. 
	\end{proof}

	In the next two lemmas we will demonstrate an association between BB nodes and good points. In particular, we will show that the association is at least one-to-one and at most $n$-to-one. This essentially allows us to conclude the proof of Theorem \ref{thm:good}.
	
	\begin{lemma}\label{lem:counting_with_intsolns}
	There exists an association $r: \mathcal{N} \rightarrow G$ that associates any internal node $N \in \mathcal{N}$ to at least one good point $x \in G$, i.e. for all $N \in \mathcal{N}$, it holds that $|r(N)| \geq 1$.
	\end{lemma}
	
	\begin{proof}
	Let $N \in \mathcal{N}$ be any internal node of $\mathcal{T}$ and $x^N$ be its optimal LP solution. Let $J \subset [n]$ denote the set of indices where $x^N$ is fractional. Define $r: \mathcal{N} \rightarrow G$ to associate $N$ with any
    $$x' \in \arg\min \{\ \pareto(x) : x \in \{0,1\}^n,\ x = x^N\ \forall j \not \in J \}.$$
    Noting that the objective function $\pareto(\cdot)$ is affine ($\lambda^*$ is fixed) and that $x^N$ is in the convex hull of the $0,1$ points $\{x \in \{0,1\}^n : x = x^N\ \forall j \not \in J \}$, it must hold that
    $$\pareto(x') \leq \pareto(x^N) \leq \gap,$$
    where the last inequality follows from Lemma \ref{lem:internal_pareto}. Then we see that
    $$x' \in \{x \in \{0,1\}^n : \pareto(x) \leq \gap\}.$$
	\end{proof}
	
	\begin{lemma}\label{lem:no_repeated_hulls}
    The association of Lemma \ref{lem:counting_with_intsolns} is at most $n$-to-$1$, i.e. for any $x \in G$, it holds that $|r^{-1}(x)| \leq n$. 
	\end{lemma}
	
	\begin{proof}
	Let $x \in G$. Let $N_1, N_2$ be two internal nodes of $\mathcal{T}$, with optimal LP solutions $x^{(1)}$ and $x^{(2)}$ respectively, such that $r(x^{(1)}) = r(x^{(2)}) = x'$. In the following paragraph we show that either $N_1$ is a descendant of $N_2$, or vice versa. Then, for any $x \in G$ all nodes associated to $x$ (i.e. all $N \in r^{-1}(x)$) must lie in the same path from the root in $\mathcal{T}$ (i.e. must all be descendants of one another). So since a path from the root in $\mathcal{T}$ can have length at most $n$, we have the desired conclusion.
	
	Suppose, for the sake of contradiction, $N_1$ and $N_2$ are not one a descendant of the other in the BB tree $\mathcal{T}$. Let $N \neq N_1, N_2$ be their lowest common ancestor in the tree $\mathcal{T}$, and let $f \in \{1,...,n\}$ be the index where node $N$ was branched on. Since nodes $N_1$ and $N_2$ are on different subtrees under $N$, without loss of generality assume that $N_1$ is in the subtree with $x_f = 0$ and $N_2$ is in the subtree with $x_f = 1$. Then since $x_f^{(1)} = 0$ and $r(x^{(1)}) = x'$, it must be that $x'_f = 0$. However, since $x^{(2)}_f = 1$, we have $x^{(2)}_f \not = x'_f$, so that by definition of $r(\cdot)$ we have $x^{(2)} \notin r^{-1}(x')$, getting the desired contradiction. 
	\end{proof}
	
	\begin{proof}[Proof of Theorem \ref{thm:good}]	
	Combining Lemma \ref{lem:counting_with_intsolns} and Lemma \ref{lem:no_repeated_hulls}, we see that the number of internal nodes of $\mathcal{T}$ (i.e. $|\mathcal{N}|$) is upper bounded by $|G| \cdot n$. Since the total number of nodes in a binary tree is at most twice the number of its internal nodes plus 1, we see that $\mathcal{T}$ has at most $2 |G| n + 1$ nodes, giving the desired bound. 
	\end{proof}
	
	We spend the remainder of this paper obtaining an upper bound of the form $n^{O(m)}$ for the number of good solutions~\eqref{eq:good}, which will then prove Theorem \ref{thm:main}.


	\section{Value of solutions and geometry of items}\label{sec:geometry}
	
	Going back to \eqref{eq:lin1}, we can see the saddle point $x^*$ as being obtained by a \emph{linear classification} of columns induced by $\lambda^*$, namely $x^*_j$ is set to 0/1 depending on whether 
	\begin{align*}
		c_j - \ip{\lambda^*}{A^j} \gtrless 0 ~~\equiv~~ \ip{(c_j, A^j)}{(1, -\lambda^*)} \gtrless 0,
	\end{align*}
	that is, depending where the column $(c_j, A^j)$ lands relative to the hyperplane in $\R^{m+1}$ $$H := \{y \in \R^{m+1} : \ip{(1,-\lambda^*)}{y} = 0\};$$ see Figure \ref{fig:slabs}.
	
		\begin{figure}
	    \centering
	    \includegraphics[scale=0.26]{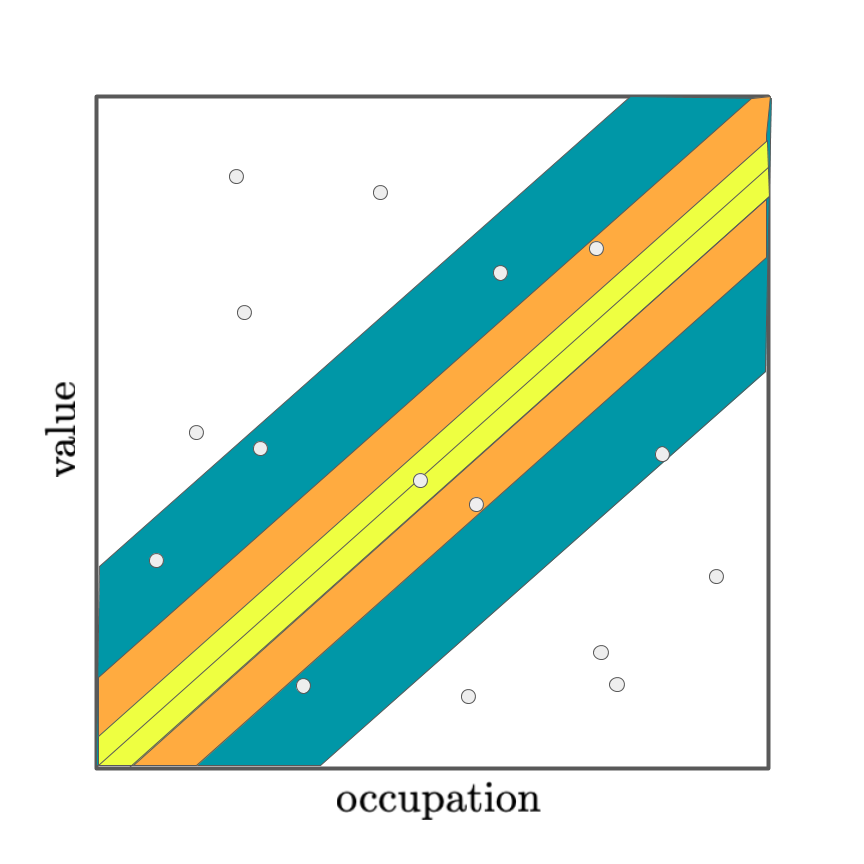}
	    \caption{Each dot represents an item/column $(c_j, A^j)$ of $\IP$. The grey line coming from the origin is $H$. The optimal solution has $x^*_j = 1$ for all items $j$ above this line, $x^*_j = 0$ for all items $j$ below this line. Each $J_\ell$ is the set of items that lie in one of the colored regions; $J_{\text{rem}}$ are the items that lie on the inner most region (yellow).}
	    \label{fig:slabs}
	\end{figure}
	
	The next lemma says that the $\pareto(x)$ of a 0/1 point $x$ increases the more it disagrees with $x^*$, and the penalty for the disagreement on item $j$ depends on the distance of the column $(c_j, A^j)$ to the hyperplane $H$. We use $d(z,U)$ to denote the Euclidean distance between a point $z$ and a set $U$, and $\ones(P)$ to denote the 0/1 indicator of a predicate $P$. This generalizes Lemma 3.1 of \cite{goldbergSpaccamela}.
	
	\begin{lemma} \label{lemma:hypDist}
		Consider any $x \in \{0,1\}^n$. Then
		\begin{align*}
			\pareto(x) \ge \sum_{j \in [n]} d\Big((c_j,A^j),\, H \Big) \cdot \ones(x_j \neq x^*_j).	
		\end{align*}
	\end{lemma}
	
	\begin{proof}
		Recalling the definition of $\pareto(\cdot)$, we can rewrite it as follows:
		\begin{align}
			\pareto(x) &= \sum_{j \in [n]} (\c_j - \ip{\lambda^*}{A^j}) (x^*_j - x_j) \label{eq:hypDist} \\
			&= \sum_{j : x^*_j = 0} (\c_j - \ip{\lambda^*}{A^j}) (x^*_j - x_j) + \sum_{j : x^*_j = 1} (\c_j - \ip{\lambda^*}{A^j}) (x^*_j - x_j).
		\end{align}  
		To analyze the right-hand side of this equation, consider a term in the first sum. Since $x^*_j = 0$, the term is non-zero exactly when the 0/1 point $x$ takes value $x_j = 1$. Moreover, since $x^*_j = 0$ implies that $\c_j - \ip{\lambda^*}{A^j}$ is negative, this means that the terms in the first sum of \eqref{eq:hypDist} equal $|\c_j - \ip{\lambda^*}{A^j}| \cdot \ones(x_j \neq x^*_j)$. A similar analysis shows that the terms in the second sum of \eqref{eq:hypDist} also equal $|\c_j - \ip{\lambda^*}{A^j}| \cdot \ones(x_j \neq x^*_j)$. Together these give
		\begin{align}
			\pareto(x) = \sum_{j \in [n]} |\c_j - \ip{\lambda^*}{A^j}| \cdot \ones(x_j \neq x^*_j). \label{eq:hypDist2}		
		\end{align}
	Finally, notice that $$|\c_j - \ip{\lambda^*}{A^j}| \ge d\Big((c_j,A^j),~H \Big),$$ because the point $(\ip{\lambda^*}{A^j}, A^j)$ belongs to the hyperplane $H(\lambda)$ we get 
	\begin{align*}
	d\Big((c_j,A^j),~H \Big) &\le \|(c_j,A^j) - (\ip{\lambda^*}{A^j}, A^j)\|_2 \\
	&= |c_j - \ip{\lambda^*}{A^j}|.
	\end{align*}
	Plugging this bound in \eqref{eq:hypDist2} concludes the proof of the lemma. 
	\end{proof}

	Thus, a good solution, i.e., one with $\pareto(x) \le \gap$, must disagree with $x^*$ on few items $j$ whose columns $(c_j, A^j)$ are ``far'' from the hyperplane $H$. To make this quantitative we bucket these distances in powers of 2, so for $\ell \ge 1$ define the set of items
	\begin{align*}
		J_{\ell} &:= \bigg\{j : d\Big((c_j,A^j),\, H \Big) \textrm{ is} \\
		&\qquad\qquad\quad \textrm{ in the interval } \Big( \tfrac{\log n}{n} 2^\ell,~ \tfrac{\log n}{n} 2^{\ell+1} \Big] \bigg\},
  	\end{align*}
  	and define $J_{\text{rem}} = [n] \setminus \bigcup_{\ell \ge 1} J_\ell$ as the remaining items (see Figure \ref{fig:slabs}). Since every item in $J_\ell$ has distance at least $\frac{\log n}{n} 2^\ell$ from $H$, we directly have the following.

	\begin{cor} \label{cor:hypDist}
		If $x \in \{0,1\}^n$ is a good point (i.e., $\pareto(x) \le \gap$), then for every $\ell \ge 1$ the number of coordinates $j \in J_\ell$ such that $x_j \neq x^*_j$ is at most $\frac{C}{2^\ell}$, where $C := \frac{n}{\log n} \cdot \gap$. 
	\end{cor}
	
	Then, an easy counting argument gives an upper bound on the total number of good points $G$. 
	
	\begin{lemma} \label{lemma:counting}
		We have the following upper bound:
		\begin{align*}
			|G| \le  2^{|J_{\text{rem}}|}\cdot \prod_{\ell = 1}^{\log C} \binom{|J_\ell|}{\le C/2^{\ell}} ,
		\end{align*}
		where $C := \frac{n}{\log n} \cdot \gap$.\footnote{We assume throughout that the $C/2^\ell$'s are integral to simplify the notation, but it can be easily checked that using $\lceil C/2^\ell \rceil$ instead does not change the results.}
	\end{lemma}
	
	\begin{proof}
	Notice that every solution in $G$ can be thought of as being created by starting with the vector $x^*$ and then changing some of its coordinates, and because of Corollary \ref{cor:hypDist} we:
		
		\begin{itemize}
			\item Cannot change the value of $x^*$ in any coordinate $j$ in a $J_\ell$ with $\ell > \log C$
		
			\item Can only flip the value of $x^*$ in at most $\frac{C}{2^\ell}$ of the coordinates $j \in J_\ell$ for each $\ell = 1,\ldots,\log C$ (recall that $x^*$ is 0/1 in all such coordinates)
			
			\item Set a new arbitrary 0/1 value for (in principle all) coordinates in $J_{\text{rem}}$.
		\end{itemize}
		Since there are at most $2^{|J_{\text{rem}}|}\cdot \prod_{\ell = 1}^{\log C} \binom{|J_\ell|}{\le C/2^{\ell}}$ options in this process, we have the desired upper bound. 
	\end{proof}
	
	Notice that, ignoring the term $2^{|J_{\text{rem}}|}$, this already gives with good probability a quasi-polynomial bound $|G| \lesssim O(n)^{\text{polylog}(n)}$ for random instances of $\IP$: the upper bound on the integrality gap from Theorem \ref{thm:gap} gives that with good probability $C \le \log n$ and so we have $\log \log n$ binomial terms, each at most $\binom{n}{\log n} \le n^{\log n}$ (since $|J_\ell| \le n$). 
	
	In order to obtain the desired polynomial bound $|G| \le n^{O(m)}$, we need a better control on $|J_\ell|$, namely the number of points at a distance from the hyperplane $H$.


	\section{Number of items at a distance from the hyperplane}\label{sec:distance} 
	
	To control the size of the sets $J_\ell$ we need to consider a random instance of $\IP$ and use the fact that the columns $(c_j, A^j)$ are uniformly distributed in $[0,1]^{m+1}$. Recalling the definition of $J_\ell$, we see that an item $j$ belongs to this set only if the column $(c_j, A^j)$ lies on the $(m+1)$-dim \emph{slab} of width $\tfrac{\log n}{n} 2^{\ell+1}$ around $H$:
	\begin{align*}
		\Big\{ y \in \R^{m+1} : d(y, H) \le \tfrac{\log n}{n} 2^{\ell+1}\Big\}.
	\end{align*}
	It can be shown that the volume of this slab intersected with $[0,1]^{m+1}$ is proportional to its width, so as long as $H$ is independent of the columns, the probability that a random column $(c_j, A^j)$ lies in this slab is $\approx \tfrac{\log n}{n} 2^{\ell+1}$. Thus, we would expect that at most $\approx (\log n)\, 2^{\ell+1}$ columns lie in this slab, which gives a much improved upper bound on the (expected) size of $J_\ell$ (as indicated above, think $\ell \le \log \log n$). Moreover, using independence of the columns $(c_j, A^j)$, standard concentration inequalities show that for \emph{each} such slab with good probability the number of columns that land in it is within a multiplicative factor from this expectation. 
	
	Unfortunately, the hyperplane $H$ that we are concerned with, whose normal $\frac{(1,-\lambda^*)}{\|(1,-\lambda^*)\|_2}$ is determined by the data, is \emph{not} independent of the columns. So we will show a much stronger \emph{uniform} bound that shows that with good probability the above phenomenon holds \emph{simultaneously} for all slabs around \emph{all} hyperplanes, which then shows that it holds for $H$. We abstract this situation and prove such uniform bound. We use $\bS^{k-1}$ to denote the unit sphere in $\R^k$. 
	
	\begin{thm}[Uniform bound for slabs] \label{thm:uniform}
		For $u \in \mathbb{S}^{k-1}$ and $w \ge 0$, define the slab of normal $u$ and width $w$ as $$S_{u, w} := \Big\{y \in \R^k : \ip{u}{y} \in [-w, w]\Big\}.$$	Let $Y^1,\ldots,Y^n$ be independent random vectors uniformly distributed in the cube $[0,1]^k$, for $n \ge k$. Then with probability at least $1 - \frac{1}{n}$, we have that for all $u \in \mathbb{S}^{k-1}$ and $w \ge \tfrac{\log n}{n}$ at most $60nw k$ of the $Y^j$'s belong to $S_{u,w}$.
	\end{thm}

	While very general bounds of this type are available (for example, appealing to the low VC-dimension of the family of slabs), we could not find in the literature a good enough such \emph{multiplicative} bound (i.e., relative to the expectation $\approx nw$). The proof of Theorem \ref{thm:uniform} instead relies on an $\e$-net type argument.
	
	This result directly gives the desired upper bound on the size of the sets $J_{\text{rem}}$ and $J_{\ell}$. 
	
	\begin{cor}\label{cor:slabs_event}
	    With probability at least $1 - \frac{1}{n}$ we have simultaneously
	    $$|J_{\text{rem}}| \leq 120(m+1) \log n $$
	    $$|J_\ell| \leq 60(m + 1)2^{\ell + 1}\log n, \quad \forall \ell \in [\log n - 1]$$
	\end{cor}

	
    \subsection{Proof of Theorem \ref{thm:uniform}} \label{sec:uniform}
    
    Let $\bS'$ be a minimal $\e$-net, for $\e := \frac{\log n}{n \sqrt{k}}$, of the sphere $\bS^{k-1}$, namely for each $u \in \bS^{k-1}$ there is $u' \in \bS'$ such that $\|u' - u\|_2 \le \e$. It is well-known that there is such a net of size at most $(\frac{3}{\e})^k$ (Corollary 4.2.13 \cite{vershyninBook}). Also define the discretized set of widths $\mathbb{W}' := \{\frac{\log n}{n}, \frac{2\log n}{n},..., \sqrt{k} + \frac{\log n}{n}\}$ so that $|\mathbb{W}'| = \frac{n \sqrt{k}}{\log n} + 1$.

 We start by focusing on a single slab in the net, observing that these slabs are determined independently of the columns. Since the vector $Y^j$ is uniformly distributed in the cube $[0,1]^k$, the probability that it belongs to a set $U \subseteq [0,1]^k$ equals the volume $\vol(U)$. Using this and an upper bound on the volume of a slab (intersected with the cube), we get that the probability that $Y^j$ lands on a slab is at most proportional to the slab's width.
    
    \begin{lemma} \label{lemma:volSlab}
 			For every slab $S_{u',w'}$, with $u' \in \bS'$ and $w' \in \mathbb{W}'$, we have $\Pr(Y^j \in S_{u',w'}) \le 2\sqrt{2} w'$.
    \end{lemma}
    
    \begin{proof}    
			Since these slabs are determined independently of the columns, it is equivalent to show that the volume $\vol(S_{u',w'} \cap [0,1]^k)$ is at most $2 \sqrt{2} w'$. Let $\text{slice}(h) := \{ y \in [0,1]^k : \ip{y}{u} = h\}$ be the slice of the cube with normal $u'$ at height $h$. It is known that every slice of the cube has $(k-1)$-dim volume $\vol_{k-1}$ at most $\sqrt{2}$~\cite{cubeSlice}, and since $S_{u',w'}\cap [0,1]^k = \bigcup_{h \in [-w',w']} \text{slice}(h)$, by integrating we get
    	\begin{align*}
			\vol(S_{u', w'} \cap [0,1]^k) = \int_{- w'}^{w'} \vol_{k-1}(\text{slice}(h)) \,\textrm{d}h ~\le~ 2 \sqrt{2} w'
    	\end{align*}
    	as desired.
    \end{proof}
    
    Let $N_{u,w}$ be the number of vectors $Y^j$ that land in the slab $S_{u,w}$. From the previous lemma we have $\E N_{u',w'} \le 2 \sqrt{2} w' n$. To show that $N_{u',w'}$ is concentrated around its expectation we need Bernstein's Inequality; the following convenient form is a consequence of Appendix A.2 of \cite{kol} and the fact $\Var(\sum_j Z_j) = \sum_j \Var(Z_j) \le \sum_j \E Z_j^2 \le \E \sum_j Z_j$ for independent random variables $Z_j$ in $[0,1]$.  
    
    \begin{lemma} \label{lemma:bernstein}
    	Let $Z_1,\ldots,Z_n$ be independent random variables in $[0,1]$. Then for all $t \ge 0$, 
			\begin{align*}
				\small
				&\Pr\bigg(\sum_j Z_j \ge \E \sum_j Z_j + t\bigg) \\
				&~~~~~\le \exp\bigg(-\min\bigg\{\frac{t^2}{4 \E \sum_j Z_j}~,~\frac{3t}{4}  \bigg\} \bigg).
			\end{align*}   	
    \end{lemma}

   \begin{lemma} \label{lemma:bernApp}
   	For each $u' \in \mathbb{S}'$ and width $w' \in \mathbb{W}'$, we have $$\Pr(N_{u',w'} \ge 20 w' nk) \le n^{- 4 k}.$$
   \end{lemma}
   
   \begin{proof}
   	Let $\mu := \E N_{u',w'}$. Notice that $N_{u',w'}$ is the sum of the independent random variables that indicate for each $j$ whether $Y^j \in S_{u',w'}$. Then applying the previous lemma with $t = 4 \mu + \frac{16}{3} k \ln n$ we get
   	\begin{align}
   		&\Pr\bigg(N_{u',w'} \,\ge\, 5 \mu + \frac{16}{3} k \ln n\bigg) \label{eq:bernApp}\\
   		&\le\, \exp\bigg( - \min\bigg\{ \frac{16}{3} k \ln n~,~\frac{3\cdot  (16/3) k \ln n}{4} \bigg\}  \bigg)\notag\\
   		&\le n^{- 4 k}, \notag
   	\end{align}
   	where in the first inequality we used $t^2 \ge (4 \mu)\cdot (\frac{16}{3} k \ln n)$. From Lemma \ref{lemma:volSlab} we get $\mu \le 2 \sqrt{2} nw'$, and further using the assumption $w' \ge \frac{\log n}{n}$ (implied by $w' \in \mathbb{W}'$) we see that
   	\begin{align*}
   	20 w' nk ~\ge~ \bigg(10\sqrt{2} + \frac{16}{3}\bigg) w' nk ~\ge~ 5 \mu + \frac{16}{3} k \ln n,
   	\end{align*}
   	and hence inequality \eqref{eq:bernApp} upper bounds $\Pr(N_{u',w'} \ge 20 w' nk)$. This concludes the proof.    	
   \end{proof}
   
   To prove the theorem we need to show that with high probability we simultaneously have $N_{u,w} \le 60 n w k$ for all $u \in \mathbb{S}^{k-1}$ and $w \ge \frac{\log n}{n}$. We associate each slab $S_{u,w}$ (for $u \in \bS^{k-1}$ and $w \in [0,\sqrt{k}]$) to a ``discretized slab'' $S_{u',w'}$ by taking $u'$ as a vector in the  net $\bS'$ so that $\|u' - u\|_2 \le \e$ and taking $w' \in \mathbb{W}'$ so that $w' \in [w + \frac{\log n}{n}, w + \frac{2 \log n}{n}]$. 
	
	\begin{lemma}
	This association has the following properties: for every $u \in \bS^{k-1}$ and $w \in [0,\sqrt{k}]$
	
	\begin{enumerate}
		\item The intersected slab $S_{u,w} \cap [0,1]^k$ is contained in the associated intersected slab $S_{u',w'} \cap [0,1]^k$. In particular, in every scenario $N_{u,w} \le N_{u',w'}$
		
		\item If the width satisfies $w \ge \frac{\log n}{n}$, then $w' \le 3 w$. 
	\end{enumerate}
	\end{lemma}
   
  \begin{proof}
  	The second property is immediate, so we only prove the first one. Take a point $y \in S_{u,w} \cap [0,1]^k$. By definition we have $\ip{y}{u} \in [-w,w]$, and also 
  	\begin{align*}
  		|\ip{y}{u'} - \ip{y}{u}| = |\ip{y}{u'-u}| &\le \|y\|_2\, \|u'-u\|_2 \\
  		&\le \sqrt{k} \e = \frac{\log n}{n},
  	\end{align*}
  	where the first inequality is Cauchy-Schwarz, and the second uses the fact that every vector in $[0,1]^k$ has Euclidean norm at most $\sqrt{k}$. Therefore $\ip{y}{u'} \in [-(w + \frac{\log n}{n}), w + \frac{\log n}{n}]$, and since the associated width satisfies $w' \ge w + \frac{\log n}{n}$ we see that $y$ belongs to $S_{u',w'} \cap [0,1]^k$. This concludes the proof. 
  \end{proof}

  \begin{proof}[Proof of Theorem \ref{thm:uniform}]
		We need to show
		\begin{align}
			\Pr\bigg[\bigvee_{u \in \bS^{k-1}, w \ge \frac{\log n}{n}} (N_{u,w} > 60 w n k) \bigg] \le \frac{1}{n}. \label{eq:uniform1}
		\end{align}
		Since there are only $n$ $Y^j$'s, we always have $N_{u,w} \le n$ and hence it suffices to consider $w \in [\frac{\log n}{n}, \frac{1}{60k}]$, in which case we can use the inequalities $N_{u,w} \le N_{u',w'}$ and $w' \le 3w$ from the previous lemma; in particular, the event $(N_{u,w} > 60 w n k)$ implies the event $(N_{u',w'} > 20 w' n k)$. Thus, using Lemma \ref{lemma:bernApp}
		\begin{align*}
			\text{LHS of } \eqref{eq:uniform1} ~&\le~ \Pr\bigg[\bigvee_{u' \in \bS', w' \in \mathbb{W}'} (N_{u',w'} > 20 w' n k) \bigg] \\
			&\le \sum_{u' \in \bS', w' \in \mathbb{W}'} \Pr(N_{u',w'} > 20 w' n k)\\
			&\le~ |\bS'|\,|\mathbb{W}'|\,n^{-4k}\\
			&\le~ \bigg(\frac{3 n \sqrt{k}}{\log n}\bigg)^{k+1} n^{-4k} ~\le~ \frac{1}{n},
		\end{align*}		
		where the last inequality uses the assumption $n \ge k$. 	This concludes the proof. 
  \end{proof}


	\section{Proof of Theorem \ref{thm:main}}\label{sec:prthm}

We can finally conclude the proof of Theorem \ref{thm:main}. 	To simplify the notation we use $O(\text{val})$ to denote $\cst \cdot \text{val}$ for some constant $\cst$. Because of Theorem \ref{thm:good}, we show that for a random instance $\I$ of $\IP$, with probability at least $1 - \frac{1}{n} - 2^{-\alpha a_2}$ the number of good points $G$ relative to this instance at most $n^{O(m + \alpha a_1 \log m)}$.
	
	For that, let $E$ be the event where all of the following hold:
	
	\begin{enumerate}
		\item Theorem \ref{thm:gap} holds for $\I$, namely $\gap(\I)$ is at most $\frac{\alpha a_1 \log^2 n}{n}$,
		\item The bound of Corollary \ref{cor:slabs_event} on the sizes of $J_{\text{rem}}$ and $J_\ell$ for all $\ell \in [\log n - 1]$, determined by the arrangement of random vectors $(c_j,A^j)$'s that comprise the columns of $\I$.
	\end{enumerate}
	By taking a union bound over these results we see that $E$ holds with probability at least $1 - \frac{1}{n} - 2^{-\alpha a_2}$. 
	
	From the first item in the definition of $E$, under $E$ we have that $C:= \frac{n}{\log n} \cdot \gap(\I)$ is at most $\alpha a_1 \log n$. Therefore, using the standard estimate $\binom{a}{\leq b} \le (4a/b)^b$ that holds for $a \ge 4b$, we get that, under $E$:
	\begin{align*}
		\prod_{\ell = \frac{\log a_1}{2}}^{\log C} \binom{|J_\ell|}{\le C/2^\ell} ~&\le~  \prod_{\ell = 1}^{\log C}  \bigg(\frac{O(m)\, 2^{2 \ell} \log n}{\alpha a_1 \log n} \bigg)^{C/2^\ell} \\
		&=~  \prod_{\ell = 1}^{\log C}  \bigg(\frac{O(m)\, 2^{2 \ell}}{\alpha a_1} \bigg)^{C/2^\ell} \\ 
		&~\le~ \bigg(\frac{O(m)}{\alpha a_1} \bigg)^{C \sum_{\ell \ge 1} \frac{1}{2^\ell}} \cdot 2^{2C \sum_{\ell \ge 1} \frac{\ell}{2^\ell}} \\
		&~\le~ \bigg(\frac{O(m)}{\alpha a_1} \bigg)^{O(C)},
	\end{align*} 
	where we started the product from $\ell = \frac{\log a_1}{2}$ to ensure we could apply the binomial estimate given only the assumption $\alpha \le 30m$; for the lower terms we can use the crude upper bound 
	\begin{align*}
		\prod_{\ell < \frac{\log a_1}{2}} \binom{|J_\ell|}{\le C/2^\ell} ~&\le~ \prod_{\ell < \frac{\log a_1}{2}} 2^{|J_\ell|}\\
		 ~&\le~ 2^{O(m a_1 \log a_1 \log n)} \\
		 &~=~ n^{O(m a_1 \log a_1)}.
	\end{align*}
	
	Plugging these bounds on Lemma \ref{lemma:counting}, we get that under $E$ the number of good points $G$ relative to the instance $\I$ is upper bounded as
	\begin{align*}
		|G| &\le 2^{O(m \log n)} \cdot n^{O(m a_1 \log a_1)}\cdot \bigg(\frac{O(m)}{\alpha a_1} \bigg)^{O(\alpha a_1 \log n)} \\
		&\le n^{O(m a_1 \log a_1 + \alpha a_1 \log m)}.
	\end{align*}
	Finally, plugging this bound on Theorem \ref{thm:good} we get that under $E$ the branch-and-bound tree for the instance $\I$ has at most 
	\begin{align*}
	&2\, n \cdot\, n^{O(m a_1 \log a_1 + \alpha a_1 \log m)} + 1\\
	&\qquad\qquad~~~\le~ n^{O(m a_1 \log a_1 + \alpha a_1 \log m)}
	\end{align*}
	nodes. This concludes the proof of Theorem \ref{thm:main}.
	

  \section{Final remarks} \label{sec:remark}
	
	We note that most of the above arguments hold not only for random problems but actually for \emph{every} 0/1 IP. In particular, Theorem \ref{thm:good} and Lemma \ref{lemma:counting} hold in such generality, which combined give the following.
	
	\begin{cor}\label{cor:final_rem}
		Consider any instance of $\text{IP}$ with arbitrary $A \in \R^{m \times n}$ and $b \in \R^m$. Then the tree of the best-bound branch-and-bound algorithm applied to this instance has at most 
		\begin{align*}
		2 \bigg( 2^{|J_{\text{rem}}|}\cdot \prod_{\ell = 1}^{\log C} \binom{|J_\ell|}{\le C/2^\ell}\bigg) n + 1
		\end{align*}
		nodes, where $C := \frac{n}{\log n} \cdot \gap$.
	\end{cor}	

	This shows that the effectiveness of branch-and-bound actually hold for every ``well-behaved'' 0/1 IP (or distributions that generate such IPs with good probability), where ``well-behaved'' means that the integrality gap must be small and there cannot be too many columns $(c_j,A^j)$ concentrated around a hyperplane (in order to control the terms $|J_\ell|$).

	
  \section*{Acknowledgments} 
  
We  are extremely grateful to Daniel Dadush for suggesting several improvements to the paper.

	    
	\bibliographystyle{plainnat}
	\bibliography{bib}


\end{document}